\documentclass[14pt]{amsart}

\usepackage{amsmath,amssymb,amsfonts,enumerate,amsthm, amscd}
\usepackage[french, ngerman, italian, english]{babel}
\usepackage[colorlinks=true]{hyperref}
\usepackage{paralist}
\usepackage{verbatim}
\usepackage{color}
\usepackage{xcolor}
\usepackage{tikz}
\usetikzlibrary{arrows,decorations.pathmorphing,backgrounds,positioning,fit}
\RequirePackage{pgfrcs}

\theoremstyle{plain}
\newtheorem{thm}{\bf Theorem}[section]

\newtheorem{lem}[thm]{\bf Lemma}
\newtheorem{cor}[thm]{\bf Corollary}

\theoremstyle{definition}
\newtheorem{defn}[thm]{\bf Definition}
\theoremstyle{remark}
\newtheorem{rem}[thm]{\bf Remark}
\newtheorem{exam}[thm]{\bf Example}

\theoremstyle{example}

\def \supp{{\mathrm{supp}}}
\def \Supp{{\mathrm{Supp}}}
\def \s{{\mathrm{star}}}
\def \lk{\mathrm{lk}}

\def \reg{\mathrm{reg}}

\def \x{\mathbf x}
\def \a{\mathbf a}
\def \b{\mathbf b}

\def \1{\mathbf 1}

\def \pd{\mathrm{pd}}

\def \Ass{\mathrm{Ass}}
\def \Min{\mathrm{Min}}

\def \pd{\mathrm{pd}}

\def \Min{\mathrm{Min}}

\def \lk{\mathrm{link}}
\def \Del{\Delta}

\def \NN{\mathbb N}
\def \ZZ{\mathbb Z}

\def \fm{\mathfrak m}

\def \F{\mathcal F}

\def \T{\mathcal T}

\begin{document}

\title[]{$k$-clean monomial ideals}
\author{Rahim Rahmati-Asghar}

\keywords{$k$-clean, monomial ideal, Stanley-Reisner ideal, Cohen-Macaulay, matroid}

\subjclass[2010]{Primary: 13F55; 13P10, Secondary: 05E40}

\begin{abstract}
In this paper, we introduce the concept of $k$-clean monomial ideals as an extension of clean monomial ideals and present some homological and combinatorial properties of them. Using the hierarchal structure of $k$-clean ideals, we show that a $(d-1)$-dimensional simplicial complex is $k$-decomposable if and only if its Stanley-Reisner ideal is $k$-clean, where $k\leq d-1$. We prove that the classes of monomial ideals like monomial complete intersection ideals, Cohen-Macaulay monomial ideals of codimension 2 and symbolic powers of Stanley-Reisner ideals of matroid complexes are $k$-clean for all $k\geq 0$.
\end{abstract}

\maketitle

\section*{Introduction}

Let $R$ be a Noetherian ring and $M$ be a finitely generated $R$-module. It is
well known that there exists a so called prime filtration
$$\F:0=M_0\subset M_1\subset \ldots\subset M_{r-1}\subset M_r=M$$
that is such that $M_i/M_{i-1}\cong R/P_i$ for some $P_i\in\Supp(M)$. We call any such
filtration of M a \textbf{prime filtration}. Set
$\Supp(\F)=\{P_1,\ldots,P_r\}$. Let $\Min(M)$ denote the set of minimal prime ideals in $\Supp(M)$. If $I$ is an ideal of $R$ then we set $\min(I)=\Min(R/I)$. Dress \cite{Dr} calls a prime filtration $\F$ of $M$ \textbf{clean} if $\Supp(\F)=\Min(M)$. The module $M$ is called clean, if $M$ admits a clean filtration and $R$ is clean if it is a clean module over itself.

Let $S=K[x_1,\ldots,x_n]$ be the polynomial ring in $n$ indeterminate over a field $K$. Let $\Del$ be a simplicial complex on the vertex set $[n]=\{1,2,\ldots,n\}$. Dress \cite{Dr} showed that $\Del$ is (non-pure) shellable in the sense of Bj\"{o}rner and Wachs \cite{BjWa}, if and only if the Stanley-Reisner ring $S/I_\Del$ is clean. The result of Dress is, in fact, the algebraic counterpart of shellability for simplicial complexes. Some subclasses of shellable complexes are $k$-decomposable simplicial complexes which were introduced by Billera and Provan \cite{BiPr} on pure simplicial complexes and then by Woodroofe \cite{Wo} on not necessarily pure ones. Simon in \cite{Si} introduced ``completed clean ideal trees'' as an algebraic counterpart of pure $k$-decomposable complexes. Actually, in the sense of Simon, the Stanley-Reisner ideal of a $k$-decomposable complex is completed clean ideal tree.

Let $I\subset S$ be a monomial ideal. We call $I$ Cohen-Macaulay (clean) if the quotient ring $S/I$ has this property. In this paper, we define the concept of $k$-clean monomial ideals. The class of $k$-clean monomial ideals are, actually, subclass of clean monomial ideals. It is the aim of this paper to study the properties of $k$-clean monomial ideals and describe relations between these ideals and $k$-decomposable simplicial complexes. Moreover, some classes of $k$-clean monomial ideals are introduced. Also, some results of \cite{BaSo,HeSoYa} are extended.

In Section 2, we introduce $k$-clean monomial ideals. We show that $k$-clean monomial ideals are clean and, also, every clean monomial ideal is $k$-clean for some $k \geq 0$ (see Theorem \ref{k-clean thm}). In Section 3, we discuss some of basic properties of $k$-clean ideals. Some homological invariants of $k$-clean monomial ideals like depth and Castelnuovo-Mumford regularity are described in this section. In the fourth section, we show that a $(d-1)$-dimensional simplicial complex $\Del$ is $k$-decomposable if and only if its associated Stanley-Reisner ideal is $k$-clean, where $k\leq d$ (see Theorem \ref{k-dec k-cl}). The last section is devoted to presenting some examples of $k$-clean monomial ideals. We show that irreducible monomial ideals and monomial complete intersection ideals are $k$-clean, for all $k\geq 0$ (see Theorems \ref{primary} and \ref{comp inter}). Then by showing that Cohen-Macaulay monomial ideals of codimension 2 (see Theorem \ref{CM cod2}) are $k$-clean, we improve Proposition 1.4. of \cite{HeSoYa}. Finally, in Theorem \ref{matroid}, we show that symbolic powers of Stanley-Reisner ideals of matroid complexes are $k$-clean for all $k\geq 0$. In this way, we improve Theorem 2.1 of \cite{BaSo}.

\section{Preliminaries}

Let $\Del$ be a simplicial complex of dimension $d-1$ with the vertex set $[n]:=\{1,2,\ldots,n\}$. Let $K$ be a field. The Stanley-Reisner monomial ideal of $\Del$ is denoted by $I_\Del$ and it is a squarefree monomial ideal in the polynomial ring $S=K[x_1,\ldots,x_n]$ generated by the monomials $\x^F=\underset{i\in F}{\prod}x_i$ which $F$ is a non-face in $\Del$. The quotient ring $S/I_\Del$ is called the \textbf{face ring} or \textbf{Stanley-Reiner ring} of $\Del$. If $\F(\Del)=\{F_1,\ldots,F_r\}$ is the set of maximal faces (facets) of $\Del$ then we set $\Del=\langle F_1,\ldots,F_r\rangle$.

For all undefined terms or notions on simplicial complexes we refer the reader to the books \cite{HeHi} or \cite{St}.

Given a simplicial complex $\Del$ on $[n]$, the \textbf{link}, \textbf{star} and \textbf{deletion} of $\sigma$ in $\Del$ are defined, respectively, by
$$\begin{array}{l}
  \lk_\Del(\sigma)=\{F\in\Del: \sigma\cap F=\emptyset, \sigma\cup F\in\Del\}, \\
  \s_\Del(\sigma)=\{ F\in\Del:\sigma\cup F\in\Del\}\ \mbox{and}\\
  \Del\backslash \sigma=\{F\in\Del:\sigma\nsubseteq F\}.
\end{array}$$
Moreover, the \textbf{Alexander dual} of $\Del$ is defined as $\Del^\vee=\{F\in\Del:[n]\backslash F\not\in\Del\}$.

Let $I\subset S$ be a squarefree monomial ideal generated by monomials of degree at least $2$. Then there exists a simplicial complex $\Del$ on $[n]$ such that $I=I_\Del$. The Alexander dual of $I$ is defined $I^\vee=I_{\Del^\vee}$.

\begin{defn}
\cite{Wo} Let $\Del$ be a simplicial complex on vertex set $[n]$.
Then a face $\sigma\in\Del$ is called a \textbf{shedding face} if
it satisfies the following property:
\begin{center}
   no facet of $(\s_\Del\sigma)\backslash\sigma$ is a facet of $\Del\backslash\sigma$.
\end{center}
\end{defn}

\begin{defn}\label{k-decom simcomplex}
\cite{Wo} A $(d-1)$-dimensional simplicial complex $\Del$ is recursively defined to be
\textbf{$k$-decomposable} if either $\Del$ is a simplex or else has a
shedding face $\sigma$ with $\dim(\sigma)\leq k$ such that both
$\lk_\Del\sigma$ and $\Del\backslash\sigma$ are
$k$-decomposable.

We consider the complexes $\{\}$ and
$\{\emptyset\}$ to be $k$-decomposable for $k\geq -1$. Also $k$-decomposability implies to $k'$-decomposability for $k'\geq k$.

A $0$-decomposable simplicial complex is called \textbf{vertex-decomposable}.
\end{defn}

We say that the simplicial complex
$\Del$ is (non-pure) \textbf{shellable} if its facets can be
ordered $F_1,F_2,\ldots,F_r$ such that, for all $r\geq 2$, the
subcomplex $\langle F_1,\ldots,F_{j-1}\rangle\cap \langle
F_j\rangle$ is pure of dimension $\dim(F_j)-1$ \cite{BjWa}. It was shown in \cite{Wo} or \cite{Jo} that a $(d-1)$-dimensional (not necessarily pure) simplicial complex $\Del$ is shellable if and only if it is $(d-1)$-decomposable.

Let $I$ be a monomial ideal of $S$. We denote by $G(I)$ the set of minimal monomial generators of $I$. Let $\min(I)$ be the set of minimal (under inclusion) prime ideals of $S$ containing $I$.

For $\a\in\NN^n$, set $\x^\a=\underset{\a(i)>0}{\prod}x^{\a(i)}_i$ and define the \textbf{support} of $\a$ by $\supp(\a)=\{i:\a(i)>0\}$. We set $\supp(\x^\a):=\supp(\a)$. Also, we define $\bar{\a}$ an $n$-tuple in $\{0,1\}^n$ with $\bar{\a}(i)=1$ if $\a(i)\neq 0$ and $\bar{\a}(i)=0$, otherwise.
Set $\nu_i(\x^\a):=\a(i)$.

Let $u,v\in S$ be two monomials. We set $[u,v]=1$ if for all
$i\in\supp(u)$, $x^{a_i}_i\nmid v$ and $[u,v]\neq 1$, otherwise.

For the monomial $u\in S$ and the monomial ideal $I\subset S$ set
\begin{center}
$I^u=\langle v\in G(I): [u,v]\neq 1\rangle$\quad and\quad $I_u=\langle
v\in G(I): [u,v]=1\rangle.$
\end{center}

\begin{defn}\label{1}
\cite{RaYa} Let $I$ be a monomial ideal with the minimal system of generators
$\{u_1,\ldots,u_r\}$. The monomial $v=x^{a_1}_1\ldots x^{a_n}_n$ is called \textbf{shedding} if $I_v\neq 0$ and for each $u_i\in G(I_v)$ and each
$l\in\supp(u)$ there exists $u_j\in G(I^v)$ such that $u_j:u_i=x_l$.
\end{defn}

\begin{defn}\label{2}
\cite{RaYa} Let $I$ be a monomial ideal minimally generated with set
$\{u_1,\ldots,u_r\}$. We say $I$ is a \textbf{$k$-decomposable}
ideal if $r=1$ or else has a shedding monomial $v$ with
$|\supp(v)|\leq k+1$ such that the ideals $I^v$ and $I_v$ are
$k$-decomposable. (Note that since the number of minimal
generators of $I$ is finite, the recursion procedure will stop.)

A $0$-decomposable monomial ideal is called \textbf{variable-decomposable}.
\end{defn}

\begin{thm}\label{Alex dual k-decom}
\cite[Theorem 2.10.]{RaYa} Let $\Del$ be a (not necessarily pure) $(d-1)$-dimensional simplicial
complex on vertex set $[n]$. Then $\Del$ is $k$-decomposable if and only if $I_{\Del^\vee}$ is $k$-decomposable, where $k\leq d-1$.
\end{thm}

\begin{defn}
\cite{MoMo} A monomial ideal $I$ is called \textbf{weakly polymatroidal} if
for every two monomials
$u=x^{a_1}_1\ldots x^{a_n}_n>_{lex}v=x^{b_1}_1\ldots x^{b_n}_n$
in $G(I)$ such that $a_1=b_1,\ldots,a_{t-1}=b_{t-1}$ and
$a_t>b_t$, there exists $j>t$ such that $x_t(v/x_j)\in I$.
\end{defn}

\begin{thm}\label{w.p. v.d.}
\cite[Theorem 4.33.]{Sh} Every weakly polymatroidal ideal $I$ is variable-decomposable.
\end{thm}

\section{$k$-clean monomial ideals}

In this section we extend the concept of cleanness introduced by Dress \cite{Dr}. Let $I\subset S$ be a monomial ideal. A prime filtration
$$\F:(0)=M_0\subset M_1\subset \ldots\subset M_{r-1}\subset M_r=S/I$$
of $S/I$ is called \textbf{multigraded}, if all $M_i$ are multigraded submodules of $S/I$, and if there are multigraded isomorphisms $M_i/M_{i-1}\cong S/P_i(-\a_i)$ with some $\a_i\in\ZZ^n$ and some multigraded prime ideals $P_i$.

A multigraded prime filtration $\F$ of $S/I$ is called \textbf{clean} if $\Supp(\F)\subseteq\min(I)$.

\begin{defn}
Let $I\subset S$ be a monomial ideal. A non unit monomial $u\not\in I$ is called a \textbf{cleaner} monomial of $I$ if $\min(I+Su)\subseteq\min(I)$.
\end{defn}

\begin{defn}
Let $I\subset S$ be a monomial ideal. We say that $I$ is \textbf{$k$-clean} whenever $I$ is a prime ideal or $I$ has no embedded prime ideals and there exists a cleaner monomial $u\not\in I$ with $|\supp(u)|\leq k+1$ such that both $I:u$ and $I+Su$ are $k$-clean.
\end{defn}

We recall the concept of ideal tree from \cite{Si}:

Let $I\subset S$ be a $k$-clean monomial ideal. By the definition, there are cleaner monomials $u_1,u_2,\ldots$ with $|\supp(u_i)|\leq k+1$ decomposing $I$. Therefore we obtain the rooted, finite, directed and binary tree $\mathcal{T}$:
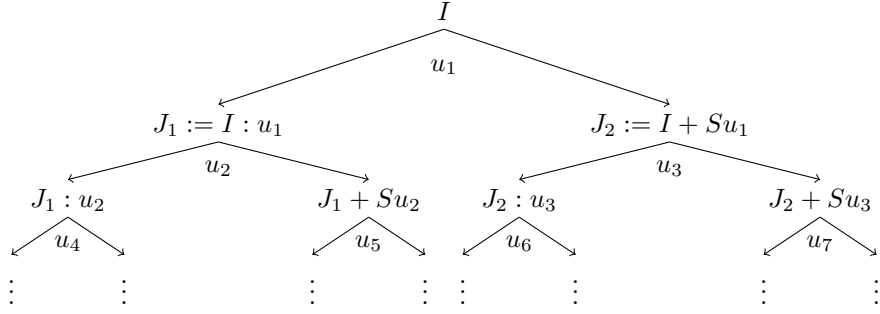
$$\begin{tikzpicture}
\coordinate (a) at (0,4);
\coordinate (b) at (-3,3);
\coordinate (b') at (-3,2.5);
\coordinate (c) at (3,3);
\coordinate (c') at (3,2.5);
\coordinate (d) at (-5,2);
\coordinate (e) at (-1,2);
\coordinate (f) at (1,2);
\coordinate (g) at (5,2);
\coordinate (d') at (-5,1.5);
\coordinate (e') at (-1,1.5);
\coordinate (f') at (1,1.5);
\coordinate (g') at (5,1.5);
\coordinate (h) at (-5.75,1);
\coordinate (i) at (-4.25,1);
\coordinate (j) at (-1.75,1);
\coordinate (k) at (-0.25,1);
\coordinate (l) at (0.25,1);
\coordinate (m) at (1.75,1);
\coordinate (n) at (4.25,1);
\coordinate (o) at (5.75,1);

\node[above] at (0,4)  {$I$};
\node at (0,3.5)  {$u_1$};
\node[below] at (-3,3)  {$J_1:=I:u_1$};
\node[below] at (3,3)  {$J_2:=I+Su_1$};
\node[below] at (-5,2)  {$J_1:u_2$};
\node[below] at (-1,2)  {$J_1+Su_2$};
\node[below] at (1,2)  {$J_2:u_3$};
\node[below] at (5,2)  {$J_2+Su_3$};
\node at (-3,2.15)  {$u_2$};
\node at (3,2.15)  {$u_3$};

\node at (-5,1.15)  {$u_4$};
\node at (-1,1.15)  {$u_5$};
\node at (1,1.15)  {$u_6$};
\node at (5,1.15)  {$u_7$};

\node[below] at (-5.75,1) {$\vdots$};
\node[below] at (-4.25,1) {$\vdots$};
\node[below] at (-1.75,1) {$\vdots$};
\node[below] at (-0.25,1) {$\vdots$};
\node[below] at (0.25,1) {$\vdots$};
\node[below] at (1.75,1) {$\vdots$};
\node[below] at (4.25,1) {$\vdots$};
\node[below] at (5.75,1) {$\vdots$};

\draw [->] (a) -- (c);
\draw [->] (a) -- (b);
\draw [->] (b') -- (d);
\draw [->] (b') -- (e);
\draw [->] (c') -- (f);
\draw [->] (c') -- (g);
\draw [->] (d') -- (h);
\draw [->] (d') -- (i);
\draw [->] (e') -- (j);
\draw [->] (e') -- (k);
\draw [->] (f') -- (l);
\draw [->] (f') -- (m);
\draw [->] (g') -- (n);
\draw [->] (g') -- (o);
\end{tikzpicture}$$
$\T$ is called the \textbf{ideal tree} of $I$ and the number of all cleaner monomials appeared in $\T$ is called the \textbf{length} of $\T$. We denote the length of $\T$ by $l(\T)$.

We define the \textbf{$k$-cleanness length} of the $k$-clean monomial ideal $I$ by $$l(I)=\min\{l(\T):\T\ \mbox{is an ideal tree of}\ I\}.$$

\begin{exam}
Consider the monomial ideal
\begin{center}
    $I=(x_1x_2x_4,x_1x_2x_5,x_1x_2x_6,x_1x_3x_5,x_1x_3x_6,x_1x_4x_5,x_2x_3x_6,$\\
    $x_2x_4x_5,x_2x_5x_6,x_3x_4x_5,x_3x_4x_6)$
\end{center}
and
\begin{center}
    $J=(x_1x_2,x_1x_3,x_1x_4)$
\end{center}
of the polynomial ring $S=K[x_1,\ldots,x_6]$. $I$ and $J$ are, respectively, $1$-clean and $0$-clean and have ideal trees $\T_1$ and $\T_2$ such that the cleaner monomials appeared in $\T_1$ and $\T_2$ are, respectively, $x_2x_3,x_1x_4,x_1x_5,x_2x_4,x_2x_5,x_2,x_1,x_3x_6,x_3$ and $x_1$.
\end{exam}

\begin{thm}\label{k-clean thm}
Every $k$-clean monomial ideal $I$ is clean. Also, every clean monomial ideal is $k$-clean for some $k\geq 0$.
\end{thm}
\begin{proof}
Let $I$ be a $k$-clean monomial ideal. We use induction on the $k$-cleanness length of $I$. Let $I$ be not prime and there exists a cleaner monomial $u\not\in I$ of multidegree $\a$ with $|\supp(u)|\leq k+1$ such that both $I:u$ and $I+Su$ are $k$-clean. By induction, $I:u$ and $I+Su$ are clean. Let
$$\F_1:I+Su=J_0\subset J_1\subset\ldots\subset J_r=S$$
and
$$\F_2:0=\frac{L_0}{I:u}\subset \frac{L_1}{I:u}\subset\ldots\subset \frac{L_s}{I:u}=\frac{S}{I:u}.$$
be clean prime filtrations and let $(L_i/I:u)/(L_{i-1}/I:u)\cong S/Q_i(-\a_i)$ where $Q_i$ are prime ideals. It is known that the multiplication map $\varphi:S/I:u(-\a)\overset{.u}{\longrightarrow}I+Su/I$ is an isomorphism. Restricting $\varphi$ to $L_i/I:u$ yields a monomorphism $\varphi_i:L_i/I:u\overset{.u}{\longrightarrow}I+Su/I$. Set $H_i/I:=\varphi_i(L_i/I:u)$. Hence $H_i/I\cong (L_i/I:u)(-\a)$. It follows that
$$\frac{H_i}{H_{i-1}}\cong \frac{H_i/I}{H_{i-1}/I}\cong \frac{(L_i/I:u)(-\a)}{(L_{i-1}/I:u)(-\a)}\cong \frac{S}{Q_i}(-\a-\a_i).$$
Therefore we obtain the following prime filtration induced from $\F_2$:
$$\F_3:I=H_0\subset H_1\subset\ldots\subset H_s=I+Su.$$
By adding $\F_1$ to $\F_3$ we obtain the following prime filtration
$$\F:I=H_0\subset H_1\subset\ldots\subset H_s=I+Su\subset J_1\subset\ldots\subset J_r=S.$$
Finally, $\Supp(\F)=\Supp(\F_1)\cup\Supp(\F_2)\subset\min(I+Su)\cup\min(I:u)\subseteq\min(I)$ and therefore $I$ is clean.

To prove the second assertion, suppose that $I$ is a clean monomial ideal. If $I$ is prime then we are done. Suppose that $I$ is not prime and let
$$\F:(0)=M_0\subset M_1\subset \ldots\subset M_{r-1}\subset M_r=S/I$$
be a clean prime filtration of $S/I$ with $M_i/M_{i-1}\cong S/P_i(-\a_i)$. We use induction on the length of the prime filtration $\F$. Since that $\Ass(S/I)\subseteq\Supp(\F)\subseteq\min(I)$, we have $\Ass(S/I)=\min(I)$. Hence $I$ has no embedded prime ideal. It follows from  Proposition 10.1. of \cite{HePo} that there is a chain of monomial ideals $I=I_0\subset I_1\subset\ldots\subset I_r=S$ and monomials $u_i$ of multidegree $\a_i$ such that $I_i=I_{i-1}+Su_i$ and $I_{i-1}:u_i=P_i$. Since that $I+Su_1$ has a clean filtration, it is $k$-clean, by induction hypothesis, where $|\supp(u_1)|\leq k+1$. On the other hand, $I+Su_1/I\cong S/P_1$. Therefore $\min(I+Su_1)=\{P_1\}\subset\min(I)$. This means that $I$ is $k$-clean.
\end{proof}



\section{Some properties of $k$-clean monomial ideals}

\begin{thm}\label{colon}
Let $I\subset S$ be $k$-clean. Then for all monomial $u\in S$, $I:u$ is $k$-clean.
\end{thm}
\begin{proof}
We use induction on the $k$-cleanness length of $I$. If $I$ is prime then $I:u$ is prime, too and we have nothing to prove. Assume that $I$ is not prime. Suppose $v$ is a cleaner monomial of $I$ with $|\supp(v)|\leq k+1$ and $I:v$ and $I+(v)$ are $k$-clean. We consider two cases:

Case 1. Let $v|u$. Then $I:u=(I:v):u/v$ and by induction hypothesis $I:u$ is $k$-clean.

Case 2. Let $v\nmid u$. We show that $v/\gcd(u,v)$ is a cleaner monomial of $I:u$. We have
$$(I:u)+(\frac{v}{\gcd(u,v)})=(I+(v)):u\quad\mbox{and}\quad (I:u):\frac{v}{\gcd(u,v)}=(I:v):\frac{u}{\gcd(u,v)}.$$
By induction, $(I:u)+(\frac{v}{\gcd(u,v)})$ and $(I:u):\frac{v}{\gcd(u,v)}$ are $k$-clean. Since $\min(I+(v))\subset\min(I)$, by some elementary computations, we obtain that $\min((I+(v)):u)\subset\min(I:u)$. Therefore $v/\gcd(u,v)$ is a cleaner monomial of $I:u$.
\end{proof}

\begin{thm}\label{radical}
The radical of each $k$-clean monomial ideal is $k$-clean.
\end{thm}
\begin{proof}
Let $I=(\x^{\a_1},\ldots,\x^{\a_r})$ be a $k$-clean monomial ideal with cleaner monomial $u=\x^\b$ with $|\supp(u)|\leq k+1$. We use induction on the $k$-cleanness length of $I$. Denote the radical of $I$ by $\sqrt{I}$. By induction hypothesis, $\sqrt{I+Su}$ and $\sqrt{I:u}$ are $k$-clean. Let $v=\x^{\supp(u)}$ and let $w$ be the product of variables $x_i$ with $i\in\supp(u)$ and $\a_j(i)>\b(i)>0$ for some $1\leq j\leq r$. $\sqrt{I}+Sv$ is $k$-clean, because $\sqrt{I}+Sv=\sqrt{I+Su}$. Also, $\sqrt{I}:v=(\sqrt{I:u}):w$ and so $\sqrt{I}:v$ is $k$-clean, by Theorem \ref{colon}. On the other hand, $\min(\sqrt{I}+Sv)\subset\min(\sqrt{I+Su})=\min(I+Su)\subset\min(I)=\min(\sqrt{I})$ and so $v$ is a cleaner monomial of $\sqrt{I}$.
\end{proof}

Let $u=x^{a_1}_{i_1}\ldots x^{a_t}_{i_t}\in S$. The \textbf{polarization} of $u$ is defined by
$$u^p=x_{i_11}\ldots x_{i_1a_1}\ldots x_{i_t1}\ldots x_{i_ta_t}.$$
If $I\subset S$ is a monomial ideal. The polarization of $I$ is a monomial ideal of $S^p=K[x_{ij}:x_{ij}|u^p\ \mbox{for some}\ u\in G(I)]$ given by $I^p=(u^p:u\in G(I))$.

Define the $K$-algebra homomorphism $\pi:S^p\rightarrow S$ by $\pi(x_{ij})=x_i$.

\begin{thm}\label{k-clean embedded}
Let $I$ be a monomial ideal with no embedded prime ideal. If $I^p$ is $k$-clean then $I$ is $k$-clean, too.
\end{thm}
\begin{proof}
We use induction on the $k$-cleanness length of $I^p$. If $I$ is a prime ideal then we have nothing to prove. Suppose that $I$ is not prime. Let $u$ be a cleaner monomial of $I^p$ with $|\supp(u)|\leq k+1$ and let $I^p:u$ and $I^p+(u)$ be $k$-clean. We claim that $\pi(u)$ is a cleaner monomial of $I$. Note that
\begin{center}
$I:\pi(u)=\pi(I^p:u)$ and $I+(\pi(u))=\pi(I^p+(u))$.
\end{center}
By induction hypothesis, $I:\pi(u)$ and $I+(\pi(u))$ are $k$-clean. Since $|\supp(\pi(u))|\leq |\supp(u)|\leq k+1$, it remains to show that $\pi(u)$ is a cleaner monomial of $I$. Let $P\in\min(I+(\pi(u)))$. Hence there exists $Q\in\min(I^p+(u))$ such that $P=\pi(Q)$. Since $Q\in\min(I^p)$, it follow that $P\in\min(I)$, as desired.
\end{proof}

\begin{lem}\label{polar cleaner}
Let $I\subset S$ be a $k$-clean monomial ideal with cleaner monomial $u$. Then $u^p$ is a cleaner monomial of $I^p$.
\end{lem}
\begin{proof}
Let $Q\in\min(I^p+(u^p))$. Then $Q\in\Ass(S^p/I^p+(u^p))$. By Corollary 2.6 of \cite{Fa1}, $\pi(Q)\in\Ass(S/I+(u))=\min(I+(u))\subset\min(I)$. Again, by Proposition 2.3 of \cite{Fa1}, $Q\in\min(I^p)$, as desired.
\end{proof}

The following theorem describes projective dimension and Castelnuovo-Mumford regularity of $k$-clean monomial ideals

\begin{thm}
Let $I\subset S$ be a $k$-clean monomial ideal with the cleaner monomial $u$. Then
\begin{enumerate}[\upshape (i)]
  \item $\pd(S/I)=\max\{\pd(S/I+(u)),\pd(S/I:u)\}$;
  \item $\reg(S/I)=\max\{\reg(S/I+(u)),\reg(S/I:u)+\deg(u)\}$.
\end{enumerate}
\end{thm}
\begin{proof}
(i) Without loss of generality we may assume that $I\subset \fm^2$. By Corollary 1.6.3. of \cite{HeHi}, $\pd(S/I)=\pd(S^p/I^p)$ and $\reg(S/I)=\reg(S^p/I^p)$. Let $\Del$ be a simplicial complex with $I_\Del=I^p$. By Lemma \ref{polar cleaner}, $u^p$ is a cleaner monomial of $I^p$. Let $u^p=\x^\sigma$ for some $\sigma\in\Del$. Therefore $\Del$ is a $k$-decomposable simplicial complex with shedding monomial $\sigma$, by Theorem \ref{k-dec k-cl}. Now it follows from Theorem 2.8 of \cite{Mo} that
$$\begin{array}{rl}
  \pd(S/I)=\pd(S^p/I_\Del)= & \max\{\pd(S^p/I_{\Del\backslash\sigma}),\pd(S^p/J_{\lk_\Del\sigma})\} \\
  = & \max\{\pd(S^p/(I+(u))^p),\pd(S^p/(I:u)^p)\} \\
  = & \max\{\pd(S/I+(u)),\pd(S/I:u)\}
\end{array}$$
where $J_{\lk_\Del\sigma}$ is the Stanley-Reisner ideal of $\lk_\Del\sigma$ considered as a complex on $V(\Del)\backslash\sigma$.

(ii) follows by a similar argument from Theorem 2.8 of \cite{Mo} and Theorem \ref{k-dec k-cl}.
\end{proof}

\begin{rem}
The concept of sequentially Cohen-Macaulayness was introduced in \cite{St} for finitely generated (graded) modules. We specially recall this concept for the quotient rings. Let $I\subset S$ be a monomial ideal. We say that $I$ is sequentially Cohen-Macaulay if there exists a finite filtration
$$\F:0=M_0\subset M_1\subset\ldots\subset M_r=S/I$$
of submodules of $S/I$ with these properties that $M_i/M_{i-1}$ is Cohen-Macaulay and
$$\dim(M_1/M_0)\leq\dim(M_2/M_1)\leq\ldots\leq\dim(M_r/M_{r-1}.$$
It was proven in \cite{HePo} that cleanness implies sequentially Cohen-Macaulayness. Therefore the class of $k$-clean monomial ideals is contained in the class of sequentially Cohen-Macaulay monomial ideals. In particular, since that every unmixed sequentially Cohen-Macaulay monomial ideal is Cohen-Macaulay, we conclude that the unmixed $k$-clean monomial ideals are Cohen-Macaulay.
\end{rem}

\section{A view toward $k$-decomposable simplicial complexes}

In this section, we prove the main result of this paper. In fact, we show that a squarefree $k$-clean monomial ideal is Stanley-Reisner ideal of a $k$-decomposable simplicial complex, and vice versa.

\begin{thm}\label{k-dec k-cl}
Let $\Del$ be a ($d-1$)-dimensional simplicial complex. Then $\sigma\in\Del$ is a shedding face of $\Del$ if and only if $\x^\sigma$ is a cleaner monomial of $I_\Del$.\par
In particular, $\Del$ is $k$-decomposable if and only if $I_\Del$ is $k$-clean, where $0\leq k\leq d-1$.
\end{thm}
\begin{proof}
We first show that $\sigma$ is a shedding face of $\Del$ if and only if $\min(I_\Del+(\x^\sigma))\subseteq\min(I_\Del)$. Since that Stanley-Reisner rings are reduced, it follows that
$$\min(I_\Del)=\{P_{F^c}:F\in\F(\Del)\}$$
and
$$\min(I_\Del+(\x^\sigma))=\{P_{F^c}:F\in\F(\Del\backslash\sigma)\}.$$
Let $\sigma$ be the shedding face of $\Del$. To show that $\x^\sigma$ is a cleaner monomial of $I_\Del$, it suffices to prove $\F(\Del\backslash\sigma)\subseteq\F(\Del)$. Suppose, on the contrary, that $F\in\F(\Del\backslash\sigma)$ and $F\subsetneqq G$ with $G\in\F(\Del)$. This implies that $\sigma\subset G$ and so $G\in\s_\Del\sigma$. On the other hand, since $F$ is a facet of $\Del\backslash\sigma$, it follows that there is $t\in\sigma$ such that $\sigma\backslash\{t\}\subset F$. We claim that $G=F\dot{\cup}\{t\}$. The inclusion ``$\supseteq$'' is clear. For the converse inclusion, if $s\in G\backslash(F\cup\{t\})$ for some $s$, then $\sigma\nsubseteq F\cup\{s\}$ and so $F\cup\{s\}\in\F(\Del\backslash\sigma)$, a contradiction. Therefore $G=F\dot{\cup}\{t\}$ and it follows that $F\in\F((\s_\Del\sigma)\backslash\sigma)$. But this contradicts the assumption that $\sigma$ is a shedding face of $\Del$. Hence $\x^\sigma$ is a cleaner monomial.

Let $\Del$ be $k$-decomposable with the shedding face $\sigma\in\Del$. By the first part, $\x^\sigma$ is a cleaner monomial of $I_\Del$. To showing that $I_\Del$ is $k$-clean, we use induction on the number of the facets of $\Del$. If $\Del$ is a simplex then the assertion is trivial. So assume that $|\F(\Del)|>1$. It is easy to check that $J_{\lk_\Del\sigma}=I_\Del:\x^\sigma$ and $I_{\Del\backslash\sigma}=I_\Del+(\x^\sigma)$. By induction hypothesis, $\lk_\Del\sigma$ and $\Del\backslash\sigma$ are $k$-decomposable if and only if $I_\Del:\x^\sigma$ and $I_\Del+(\x^\sigma)$ are $k$-clean. Therefore $I_\Del$ is $k$-clean.

The reverse directions of both parts follow easily in similar arguments.
\end{proof}

\begin{rem}\label{exam k-cl}
Note that a $k$-clean monomial ideal need not be $k'$-clean for $k'<k$. Consider the monomial ideal $I\subset K[x_1,\ldots,x_6]$ with the minimal generator set
\begin{center}
$G(I)=\{x_1x_2x_4,x_1x_2x_5,x_1x_2x_6,x_1x_3x_5,x_1x_3x_6,x_1x_4x_5,$\\
$x_2x_3x_6,x_2x_4x_5,x_2x_5x_6,x_3x_4x_5,x_3x_4x_6\}$.
\end{center}
$I$ is the Stanley-Reisner ideal of the simplicial complex $$\Del=\langle 124,125,126,135,136,145,236,245,256,345,346\rangle$$
on $[6]$. It was shown in \cite{Si} that $\Del$ is shellable but not vertex-decomposable. It follows from Theorem \ref{k-dec k-cl} that $I$ is clean but not $0$-clean. To see more examples of clean ideals which are not $0$-clean we refer the reader to \cite{Ha,MoTa}.
\end{rem}

\begin{rem}
Let $I$ be a clean monomial ideal and $\dim(S/I)=d$. By Theorem \ref{k-clean thm}, $I$ is $k$-clean for some $k\geq 0$ with cleaner monomial $u$. It follows from Theorem \ref{radical} that $\sqrt{I}$ is $k$-clean with cleaner monomial $v=\x^{\supp(u)}$. Let $I_\Del=\sqrt{I}$ for some simplicial complex $\Del$ on $[n]$. By Theorem \ref{k-dec k-cl}, we have $|\supp(u)|=|\supp(v)|\leq \dim(\Del)+1=d$. Therefore $I$ is $(d-1)$-clean.

On the other hand, every $k$-clean monomial ideal is also $(k+1)$-clean. This means that the $k$-cleanness is a hierarchical structure. Therefore we have the following implications:

\begin{center}
$0$-clean $\Rightarrow$ $1$-clean $\Rightarrow$ \ldots $\Rightarrow$ ($d-1$)-clean $\Leftrightarrow$ clean.
\end{center}
In Remark \ref{exam k-cl} we implied that above implications are strict.
\end{rem}

\begin{cor}\label{k-clean k-decom}
Let $I\subset S$ be a squarefree monomial ideal generated by monomials of degree at least $2$. Then $I$ is $k$-clean if and only if $I^\vee$ is $k$-decomposable.
\end{cor}
\begin{proof}
Let $\Del$ be a simplicial complex on $[n]$ such that $I=I_\Del$. The assertion follows from Theorems \ref{k-dec k-cl} and \ref{Alex dual k-decom}.
\end{proof}

\section{Some classes of $k$-clean ideals}

In this section, we introduce some classes of $k$-clean monomial ideals.

\subsection{Irreducible monomial ideals}

\begin{thm}\label{primary}
Every irreducible monomial ideal is $0$-clean.
\end{thm}
\begin{proof}
Let $I$ be a irreducible monomial ideal. We want to show that $I$ is $0$-clean. By Theorem 1.3.1. of \cite{HeHi}, $I$ is generated
by pure powers of the variables. Without loss of generality we may assume that $I=(x^{a_1}_1,\ldots,x^{a_m}_m)$ with $a_i\neq 0$ for all $i$. We use induction on $\sum^m_{i=1}a_i$. If $\sum^m_{i=1}a_i=m$, then $I$ is prime and we are done. Suppose that $\sum^m_{i=1}a_i>m$. So we can assume that $a_1>1$. We have
\begin{center}
    $I:x_1=(x^{a_1-1}_1,x^{a_2}_2,\ldots,x^{a_m}_m)$ and $I+(x_1)=(x_1,x^{a_2}_2,\ldots,x^{a_m}_m)$.
\end{center}
By induction hypothesis, $I:x_1$ and $I+(x_1)$ are $0$-clean. Clearly, $x_1$ is a cleaner monomial and so the proof is completed.
\end{proof}

\subsection{Monomial complete intersection ideals}

\begin{thm}\label{comp inter}
Let $I\subset S$ be a monomial complete intersection ideal. Then $S/I$ is $0$-clean.
\end{thm}
\begin{proof}
Let $G(I)=\{M_1,\ldots,M_r\}$. By the assumption $M_1,\ldots,M_r$ is a regular sequence. Hence $\gcd(M_i,M_j)=1$ for all $i\neq j$. If $I$ is a primary ideal then we are done, by Theorem \ref{primary}. Suppose that $I$ is not primary. We use induction on $n$ the number of variables. Let $|\supp(M_1)|>1$ and let $\nu_1(M_1)=a$. Then
\begin{center}
    $I:x^a_1=(M_1/x^a_1,M_2,\ldots,M_r)$ and $I+(x^a_1)=(x^a_1,M_2,\ldots,M_r)$.
\end{center}
Since that $(M_1/x^a_1,M_2,\ldots,M_r)$ and $(x^a_1,M_2,\ldots,M_r)$ are complete intersection monomial ideals with the number of variables less that $n$, we deduce that $I:x^a_1$ and $I+(x^a_1)$ are $0$-clean, by induction hypothesis. Set $J:=(M_2,\ldots,M_r)$. Since that
$$\min(I+(x^a_1))=\{P+(x_1):P\in\min(J)\}$$
and
$$\min(I)=\{P+(x_i):P\in\min(J)\ \mbox{and}\ x_i|M_1\}.$$
we conclude that $\min(I+(x^a_1))\subset\min(I)$ and so $x^a_1$ is a cleaner monomial.
\end{proof}

\subsection{Cohen-Macaulay monomial ideals of codimension $2$}

Proposition 2.3 from \cite{HeHiZh} says that if $I\subset S$ is a squarefree monomial ideal with 2-linear resolution, then after suitable renumbering of the variables, one has the following property:
\begin{center}
if $x_ix_j\in I$ with $i\neq j$, $k>i$ and $k>j$, then either $x_ix_k$ or $x_jx_k$ belongs to $I$.
\end{center}
Let $I$ has a 2-linear resolution and the monomials in $G(I)$ be ordered by the lexicographical
order induced by $x_n>x_{n-1}>\ldots>x_1$. Let $u=x_sx_t>v=x_ix_j$ be squarefree monomials in $G(I)$ with $s<t$
and $i<j$. We have $t\geq j$. If $t=j$, then $x_s(v/x_i)=u\in G(I)$. If $t>j$ then by the above property either $x_ix_t\in G(I)$ or $x_jx_t\in G(I)$. This immediately implies the following lemma.

\begin{lem}\label{linear res}
If $I$ is a squarefree monomial ideal generated in degree 2 which has a linear resolution, then after suitable
renumbering of the variables, $I$ is weakly polymatroidal.
\end{lem}

\begin{thm}\label{CM cod2}
Let $I\subset S$ be a monomial ideal which is Cohen-Macaulay and of codimension $2$. Then $S/I$ is $0$-clean.
\end{thm}
\begin{proof}
Since $I$ has no embedded prime ideals, if we show that $I^p$ is $0$-clean then it follows from Theorem \ref{k-clean embedded} that $I$ is $0$-clean. Let $\Del$ be a simplicial complex with $I_\Del=I^p$. Since $I$ is Cohen-Macaulay, by Corollary 1.6.3. of \cite{HeHi}, $I_\Del$ is Cohen-Macaulay, too. In particular, $I_\Del^\vee$ has linear resolution, by the Eagon-Reiner theorem \cite{EaRe}. It follows from Lemma \ref{linear res} and Theorems \ref{w.p. v.d.} and \ref{k-dec k-cl} that $I_\Del=I^p$ is $0$-clean, as desired.
\end{proof}

\subsection{Symbolic powers of Stanley-Reisner ideals of matroid complexes}

Let $\Del$ be a simplicial complex and let $I^{(m)}_\Del$ denote the $m$th symbolic power of $I_\Del$. Minh and Trung \cite{MiTr} and Varbaro \cite{Va} independently proved that $\Del$ is a matroid if and only if $I^{(m)}_\Del$ is Cohen-Macaulay for all $m\in\NN$. Later, in \cite{TeTr}, Terai and Trung showed that $\Del$ is a matroid if and only if $I^{(m)}_\Del$ is Cohen-Macaulay for some integer $m\geq 3$. Recently, Bandari and Soleyman Jahan \cite{BaSo} proved that if $\Del$ is a matroid, then $I^{(m)}_\Del$ is clean for all $m\in\NN$. In this section, we improve this result by showing that if $\Del$ is a matroid, then $I^{(m)}_\Del$ is $0$-clean for all $m\in\NN$.

\begin{thm}\label{matroid}
Let $\Del$ be a matroid complex with $I=I_\Del$. Then for all $m\geq 1$, $I^{(m)}$ is $0$-clean.
\end{thm}
\begin{proof}
Let $\Del=\langle F_1,\ldots,F_t\rangle$. Then $I=I_\Del=\bigcap^t_{i=1}P_{F^c_i}$ and $(I_\Del)^{(m)}=\bigcap^t_{i=1}(P_{F^c_i})^{(m)}$. Since $\Del$ is a matroid and $I$ is Cohen-Macaulay, it follows that $I^{(m)}$ has no embedded prime ideal. Therefore if we show that $(I^{(m)})^p$ is $0$-clean then the proof is completed, by Theorem \ref{k-clean embedded}.

In \cite{BaSo} the authors introduced an ordering on the variables of $S^p$ and showed that $((I^{(m)})^p)^\vee$ has linear quotients with respect to this ordering. We improve this result by considering the same ordering to show that $((I^{(m)})^p)^\vee$ is weakly polymatroidal. Then by Theorem \ref{w.p. v.d.} and Corollary \ref{k-clean k-decom}, $(I^{(m)})^p$ is $0$-clean. We use some notations of the proof of \cite[Theorem 2.1.]{BaSo}. It is known that $\Del^c$ is a matroid. Let $\dim(\Del^c)=r-1$. We  set $J=((I^{(m)})^p)^\vee$. Then
$$G(J)=\{x_{i_1,j_1}x_{i_2,j_2}\ldots x_{i_r,j_r}:\{i_1,\ldots,i_r\}\ \mbox{is a facet of}\ \Del^c\}$$
where $1\leq j_l\leq m$ and $\sum^r_{l=1}j_l\leq m+r-1$.

Consider the order $<$ on the variables of $S^\alpha$ by setting $x_{i,j}>x_{i',j'}$ if either $j<j'$, or $j=j'$ and $i<i'$. Let $u,v\in G(J)$ with $u=x_{i_r,j_r}\ldots x_{i_2,j_2} x_{i_1,j_1}>v=x_{i_r',j_r'}\ldots x_{i_2',j_2'} x_{i_1',j_1'}$ such that $x_{i_l,j_l}=x_{i_l',j_l'}$ for all $l>t$ and $x_{i_t,j_t}>x_{i_t',j_t'}$. We have two cases:

\textbf{Case 1.} $x_{i_t}|x_{i_r'}\ldots x_{i_{t+1}'}x_{i_t'}$. Let $i_l'=i_t$. It is clear that $j_t<j_l'$. In particular, $x_{i_t,j_t}(v/x_{i_l',j_l'})\in G(J)$.

\textbf{Case 2.} $x_{i_t}\nmid x_{i_r'}\ldots x_{i_{t+1}'}x_{i_t'}$. Since $I_{\Del^\vee}$ is matroidal, it follows from \cite[Lemma 3.1.]{HeHi6} that there exists $i'_l\not\in\{i_1,\ldots,i_r\}$ such that $x_{i_t}(x_{i_r'}\ldots x_{i_1'}/x_{i_l'})\in I_{\Del^\vee}$. Therefore $$x_{i_r',j_r'}\ldots x_{i_{l-1}',j_{l-1}'}x_{i_{l+1}',j_{l+1}'}\ldots x_{i_t,j_t}x_{i_{t-1}',j_{t-1}'}\ldots x_{i_1',j_1'}\in G(J).$$

Therefore $J$ is weakly polymatroidal, as desired.
\end{proof}

It follows from Theorem \ref{matroid} that we can add the condition ``$0$-cleanness of $I^{(m)}_\Del$ for all $m>0$'' to \cite[Corollary 2.3.]{BaSo}:

\begin{cor}
Let $\Del$ be a pure simplicial complex and $I=I_\Del\subset S$. Then the following
conditions are equivalent:
\begin{enumerate}[\upshape (i)]
  \item $\Del$ is a matroid;
  \item $S/I^{(m)}$ is $0$-clean for all integer $m>0$;
  \item $S/I^{(m)}$ is clean for some integer $m>0$;
  \item $S/I^{(m)}$ is clean for some integer $m\geq 3$;
  \item $S/I^{(m)}$ is Cohen-Macaulay for some integer $m\geq 3$;
  \item $S/I^{(m)}$ is Cohen-Macaulay for all integer $m>0$.
\end{enumerate}
\end{cor}

Cowsik and Nori in \cite{CoNo} proved that for any homogeneous radical ideal $I$ in the polynomial ring $S$, all the powers of $I$ are Cohen-Macaulay if and only if $I$ is a complete intersection. We call the simplicial complex $\Del$ \textbf{complete intersection} if $I_\Del$ is a complete intersection ideal. Therefore the simplicial complex $\Del$ is a complete intersection if and only if $I^m_\Del$ is Cohen-Macaulay for any $m\in\NN$ (\cite[Theorem 3]{TeYo}). We improve this result in the following. By the fact that if $I^m_\Del$ is Cohen-Macaulay then $I^m_\Del$ is equal to the $m$th symbolic power $I^{(m)}_\Del$ of $I_\Del$ we have

\begin{cor}
Let $\Del$ be a pure simplicial complex and $I=I_\Del\subset S$. Then the following
conditions are equivalent:
\begin{enumerate}[\upshape (i)]
  \item $\Del$ is a complete intersection;
  \item $S/I^m$ is $0$-clean for all integer $m>0$;
  \item $S/I^m$ is clean for some integer $m>0$;
  \item $S/I^m$ is clean for some integer $m\geq 3$;
  \item $S/I^m$ is Cohen-Macaulay for some integer $m\geq 3$;
  \item $S/I^m$ is Cohen-Macaulay for all integer $m>0$.
\end{enumerate}
\end{cor}

\textbf{Acknowledgments:}
The author would like to express his sincere gratitude to the referee for his/her helpful comments that helped to improve the quality of the manuscript. The work was supported by the research council of the University of Maragheh.

\ \\ \\
Rahim Rahmati-Asghar,\\
Department of Mathematics, Faculty of Basic Sciences,\\
University of Maragheh, P. O. Box 55181-83111, Maragheh, Iran.\\
E-mail:  \email{rahmatiasghar.r@gmail.com}

\end{document}